\documentclass[11pt]{amsart}

\usepackage{color}
\usepackage{amstext}
\usepackage{amsthm}
\usepackage{amsmath}
\usepackage{amssymb}
\usepackage{latexsym}
\usepackage{amsfonts}
\usepackage{graphicx}
\usepackage{graphpap}

\usepackage[OT2,T1]{fontenc}
\DeclareSymbolFont{cyrletters}{OT2}{wncyr}{m}{n}
\DeclareMathSymbol{\Sha}{\mathalpha}{cyrletters}{"58}

\def\bb{{\mathcal B}}

\def\dd{{\mathcal D}}

\def\kk{{\mathcal K}}
\def\ll{{\mathcal L}}

\def\ss{{\mathcal S}}

\def\eps{\varepsilon}
\def\dst{\displaystyle}

\renewcommand{\Im}{\mathrm{Im}\,}

\DeclareMathOperator{\supp}{supp}

\DeclareMathOperator{\sgn}{sgn}

\def\R{{\mathbb{R}}}
\def\S{{\mathbb{S}}}

\def\Z{{\mathbb{Z}}}

\def\d{\,{\mathrm{d}}}
\newcommand{\norm}[1]{{\left\|{#1}\right\|}}

\newcommand{\abs}[1]{{\left|{#1}\right|}}
\newcommand{\scal}[1]{{\left\langle{#1}\right\rangle}}

\newcommand{\pv}{\operatorname{p.v.}}

\newcommand{\un}{\mathbf{1}}

{\end{list}}

\numberwithin{equation}{section}

\newtheorem{thm}{Theorem}[section]
\newtheorem{lm}[thm]{Lemma}

\newtheorem*{prop*}{Proposition}

\theoremstyle{remark}

\newtheorem*{rem*}{Remark}

\title{Lower bounds for the Dyadic Hilbert transform}

\author{Philippe Jaming}
\address{Philippe Jaming, Univ. Bordeaux, IMB, UMR 5251, F-33400 Talence, France.
CNRS, IMB, UMR 5251, F-33400 Talence, France.}
\email{Philippe.Jaming@u-bordeaux.fr}

\author{Elodie Pozzi}
\address{Elodie Pozzi, Univ. Bordeaux, IMB, UMR 5251, F-33400 Talence, France.
CNRS, IMB, UMR 5251, F-33400 Talence, France.}
\email{elodie.pozzi@math.u-bordeaux1.fr}

\author{Brett D. Wick}
\address{Brett D. Wick, Department of Mathematics\\ Washington University -- St. Louis\\ One Brookings Drive\\ St. Louis, MO USA 63130-4899}
\email{wick@math.wustl.edu}


\begin{document}

\begin{abstract} 
In this paper, we seek lower bounds of the dyadic Hilbert transform (Haar shift) of the form $\norm{\Sha f}_{L^2(K)}\geq C(I,K)\norm{f}_{L^2(I)}$
where $I$ and $K$ are two dyadic intervals and $f$ supported in $I$. If $I\subset K$ such bounds exist while in the other cases $K\subsetneq I$ and $K\cap I=\emptyset$ such bounds are only available under additional constraints on the derivative of $f$. 
In the later case, we establish a bound of the form $\norm{\Sha f}_{L^2(K)}\geq C(I,K)|\scal{f}_I|$ where $\scal{f}_I$
is the mean of $f$ over $I$. This sheds new light on the similar problem for the usual Hilbert transform.\\[3pt]
{\bf R\'esum\'e} Dans cet article, nous \'etablissons des bornes pour la transform\'ee de Hilbert dyadique ({\it Haar shift}) de la forme $\norm{\Sha f}_{L^2(K)}\geq C(I,K)\norm{f}_{L^2(I)}$ o\`u $I$ et $K$ sont des intervalles dyadiques et $f$ est \`a support dans $I$. Si $I\subset K$ de telles bornes existent sans condition suppl\'ementaire sur $f$ alors que dans les cas $K\subsetneq I$ et $K\cap I=\emptyset$ une telle borne n'existe que si on impose une condition sur la d\'eriv\'ee de $f$. Dans le dernier cas nous \'etablissons une borne de la forme $\norm{\Sha f}_{L^2(K)}\geq C(I,K)|\scal{f}_I|$ o\`u $\scal{f}_I$ est la moyenne de $f$ sur $I$. Ce travail permet ainsi une meilleure compr\'ehension du probl\`eme similaire pour la transform\'ee de Hilbert sur $\R$.
\end{abstract}

\maketitle

\section{Introduction}

The aim of this paper is to establish lower bounds on the dyadic Hilbert transform (Haar shift) in the spirit of those that are known for the usual Hilbert transform.

The Hilbert transform is one of the most ubiquitous and important operators in harmonic analysis.
It can can be defined on $L^2(\R)$ as the Fourier multiplier $\widehat{Hf}(\xi)=-i\sgn(\xi)\widehat{f}(\xi)$ which shows that
$H\,: L^2(\R)\to L^2(\R)$ is a unitary bijection. Alternatively, the Hilbert transform is defined via
$$
Hf(x)=\frac{1}{\pi}\pv\int_\R\frac{f(y)}{x-y}\,\mbox{d}y.
$$
While boundedness of this operator is by now rather well understood, obtaining lower bounds for the truncated Hilbert transform
is still an ongoing task. More precisely, we are looking for bounds of the form $\norm{\un_KHf}_{L^2(\R)}\gtrsim \norm{f}_{L^2(\R)}$
(for some set $K\subsetneq\R$ and $f$ satisfying some additional constraint).
Without additional constraints, such an inequality
can of course not hold and a first restriction one usually imposes is that $f$ is supported in some interval $I$.
Before describing existing literature, let us first motivate the question.

The most well known application of the Hilbert transform comes from complex analysis. Indeed, if $F$ is a reasonably decaying holomorphic function on the upper half-plane, then 
its boundary value $f$ satisfies $Hf=-if$. In particular, its real and imaginary
parts are connected via $\Im(f)=H\Re(f)$ and $\Re(f)=-H\Im(f)$. Conversely, if $f$ is a reasonable real valued function, say
$f\in L^2(\R)$ with $\supp f\subset I$, $I$ some interval, then $\tilde f:=f+iHf$ is the boundary value of a holomorphic function
in the upper half-plane. The question we are asking is whether the knowledge of $\Im(\tilde f)$ on some interval $K$
determines $f$ stably. In other words, we are looking for an inequality of the form 
$\norm{\Im(\tilde f)}_{L^2(K)}\gtrsim\norm{\Re(\tilde f)}_{L^2(I)}$.

An other instance of the Hilbert transform is in the inversion formula of the Radon Transform. 
Recall from \cite[Chapter II]{Na} that the Radon transform of a function $f\in\ss(\R^2)$ is defined by
$$
Rf(\theta,s)=\int_{\scal{x,\theta}=s}f(x)\,\mbox{d}x,\qquad \theta\in \S^1,s\in\R
$$
while the inversion formula reads
$$
f(x)=\frac{1}{4\pi}\int_{S^1}H_s[\partial_s Rf(\theta,\cdot)](\theta,\scal{x,\theta})\,\mbox{d}\sigma(\theta)
$$
where the Hilbert transform acts in the $s$-variable. In practice, $Rf(\theta,s)$ can only be measured for $s$ in 
a given interval $K$ which may differ from the relevant interval for $f$. This is a second (and main) motivation 
for establishing lower bounds on the Hilbert transform which should lead to estimates of stable invertibility
of the restricted view Radon transform. The introduction of \cite{CNDK} provides nice insight on this issue.

It turns out that the relative position of the intervals $I$ and $K$ plays a central role here and we distinguish four
cases:

\begin{itemize}
\item {\sl Covering.} When $K\supsetneq I$ the inversion is stable and an explicit inversion formula is known \cite{Tr}.

\item {\sl Interior problem.} When $K\subset I$, stable reconstruction is no longer possible. This case, known as the {\em interior
problem} in tomography has been extensively studied ({\it see e.g.} \cite{CNDK,Ka2,KKW,KCND,YYW}).

\item {\sl Gap.} When $I\cap K=\emptyset$, the singular value decomposition of the underlying operator has been given in \cite{Ka1}
and this case was further studied by Alaifari, Pierce, and Steinerberger in \cite{APS}. It turns out that oscillations of $f$ imply instabilities of the problem.
The main result of \cite{APS} is that there exists constants $c_1,c_2$ depending only on $I,K$ such that, for every $f\in H^1(I)$,
$$
\norm{Hf}_{L^2(K)}\geq c_1\exp\left(-c_1\frac{\norm{f'}_{L^2(I)}}{\norm{f}_{L^2(I)}}\right)\norm{f}_{L^2(I)}.
$$
Moreover, the authors conjecture that $\norm{f'}_{L^2(I)}$ may be replaced by $\norm{f'}_{L^1(I)}$.

\item {\sl Overlap.} When $I\cap K\not=\emptyset$ and $I\cap (\R\setminus K)\not=\emptyset$, a pointwise stability estimate
has been shown in \cite{DNCK} while the spectral properties of the underlying operator
are the subject of \cite{AK,ADK}.
\end{itemize}

Most proofs go through spectral theory. More precisely, the strategy of proof is the same as for the
similar problem for the Fourier transform. Recall that in their seminal work on time-band limiting, 
Landau, Pollak, Slepian found a differential operator that commutes with the ``time-band'' limiting operator
({\it see } \cite{Sl} for an overview of the theory and further references).
The spectral properties of this differential operator are relatively easy to study and the spectral
properties of the ``time-band'' limiting operator then follow.
The counter-part of this strategy is that it relies on a ``happy accident'' (as termed by Slepian)
that does not shed light on the geometric/analytic features at play in the Hilbert transform.
Therefore, no hint towards lower bounds for more general Calderon-Zygmund operators, nor towards the conjecture
in \cite{APS} is obtained through that approach.

Our aim here is precisely to shed new light on lower bounds for the truncated Hilbert transform. To do so, we follow
the current paradigm in harmonic analysis by replacing the Hilbert transform by its dyadic version (Haar shift) which
serves at first as a toy model. We then study the {\it gap, covering and interior} problems for the Haar shift.

To be more precise, let $\dd$ be the set of dyadic intervals. To a dyadic interval $I$, we associate the Haar function
$h_I=|I|^{-1/2}(\un_{I_+}-\un_{I_-})$ where $I_\pm$ are the sons of $I$
and $|I|$ its the length. The dyadic Hilbert transform (Haar shift) is defined by
$$
\Sha f=\sum_{I\in\dd}\scal{f,h_I}\Sha h_I
$$
where $\Sha h_I=2^{-1/2}(h_{I_+}-h_{I_-})$ (see the beginning of the next section for more details). One can define a similar transform for generalized dyadic intervals obtained by dilating and 
properly translating $\dd$. It turns out that the usual Hilbert transform is the average over a suitable family of generalized dyadic 
intervals of the corresponding Haar shifts, see \cite{Pe,Pe1, Hy}. This approach has been very successful for upper bounds but it seems 
much less adapted to lower bounds; though we point to two cases in \cite{NV,NRVV} where lower estimates for the martingale transforms are obtained and provide related lower estimates for the Hilbert transform.

Nevertheless, the Haar shift shares many common features with the continuous Hilbert transform,
and this is why we here establish lower bounds for this transform. We hope those lower bounds give some insight on the problem of 
establishing lower bounds for the truncated Hilbert transform. However, our results depend heavily on the particular structure of
the Haar shift we consider. It would be interesting and probably challenging to extend our computations to general shifts and in 
particular to Haar multipliers of fixed sign pattern.  Since we are dealing with a very particular Haar shift we are able to obtain precise formulas and estimates by direct computations, see e.g. equation \eqref{eq:lbunkHf} below.  It would be interesting to establish similar formulas for general dyadic shifts as defined in \cite{Hy2}.

The main result we obtain is the following:

\medskip

\noindent{\bf Theorem.} {\sl Let $I,K$ be two dyadic intervals. Then

\begin{enumerate}
\item {\em Covering.} If $I\subset K$ then $\dst\norm{\un_K\Sha f}_2\geq \frac{1}{2}\norm{f}_2$ for every $f\in L^2(\R)$
with $\supp f\subset I$.

\item {\em Gap.} If $I\cap K=\emptyset$, then no estimate of the form $\dst\norm{\un_K\Sha f}_2\gtrsim \norm{f}_2$ holds
for every $f\in L^2(\R)$ with $\supp f\subset I$. But

-- either $I\subset[2^{M-1},2^M]$ and $K\subset[0,2^{M-2}]$ for some integer $M$, then $\un_K\Sha f=0$
for every $f\in L^2(\R)$
with $\supp f\subset I$

-- or for every $0<\eta<1$, there exists $C=C(I,K,\eta)$ such that
$\dst\norm{\un_K\Sha f}_2\geq C\norm{f}_2$ for every $f\in L^2(\R)$ of the form $f=f_0\un_I$ with $f_0\in W^{1,2}(\R)$
and $|I|\|f_0^\prime\|_{L^2(I)}\leq 2\pi\eta\|f_0\|_{L^2(I)}$.

\item {\em Interior problem.} If $K\subset I$, then no estimate of the form $\dst\norm{\un_K\Sha f}_2\gtrsim \norm{f}_2$ holds
for every $f\in L^2(\R)$ with $\supp f\subset I$. But $\dst\norm{\un_K\Sha f}_2\geq \norm{\un_K f}_2$ for every $f\in L^2(\R)$
with $\supp f\subset I$.
\end{enumerate}
}

\medskip

Note that the fact that we assume that both $I,K$ are dyadic implies that the overlapping case does not occur here.
In the Gap case, we actually show that $\un_K\Sha f=\dst C(I,K)\int_If(x)\,\mbox{d}x$. Therefore, if $f$ has zero mean,
then its Haar shift is zero outside its support. This is a major difference with the Hilbert transform which only has extra decay 
in that case. As a consequence, one can not recover functions with zero-mean from their Haar shift outside the support.
To avoid this situation,
one may use the Poincar\'e-Wirtinger inequality to control the mean of $f$ by its $L^2$-norm when $f$ has small derivative.

In Section \ref{s2} we collect basic facts and notation and Sections \ref{s3}, \ref{s4}, and \ref{s5} are then devoted
each to one of the cases that arise in our main theorem. 

\medskip


\section{Notations and Computations of Interest}
\label{s2}

In this paper, all functions will be in $L^2(\R)$. We write
$$
\norm{f}_{L^2}=\left(\int_{\R} |f(x)|^2\d x\right)^{1/2}\quad,\quad
\scal{f,g}_{L^2}=\int_{\R} f(x)\overline{g(x)}\d x.
$$
For $I$ an interval of finite length $|I|$ and $f\in L^2(\R)$, we write
$$
\left\langle f\right\rangle_I=\frac{1}{|I|}\int_I f(x)\d x
$$  
for the mean of $f$ over $I$.

Let $\mathcal{D}$ denote the collection of dyadic intervals on $\mathbb{R}$, namely the intervals of the form 
$\dd=\{[2^k\ell,2^k(\ell+1)): k,\ell\in\Z\}$. For $I=[2^k\ell,2^k(\ell+1))$, we denote the children of $I$ by 
$I_-=[2^k\ell,2^k(\ell+1/2))=[2^{k-1}2\ell,2^{k-1}(2\ell+1))\in\dd$
and $I_+=[2^k(\ell+1/2),2^k(\ell+1))=[2^{k-1}(2\ell+1),2^{k-1}(2\ell+2))\in\dd$.  The parent of $I$, denoted $\widehat{I}$, is the unique interval in $\dd$ such that $I=\widehat{I}_{\eps(I)}$ with $\eps(I)\in\{\pm 1\}$.\\

We will frequently use the following computations: If $\ll\in\dd$, then
\begin{equation}
\label{eq:freqcomp1}
\sum_{L\in\dd,L\supsetneq\ll}\frac{1}{|L|}=\frac{1}{|\ll|}\sum_{k=1}^\infty 2^{-k}=
\frac{1}{|\ll|}
\end{equation}
while for $\ll\subsetneq\kk\in\dd$
\begin{equation}
\label{eq:freqcomp2}
\sum_{L\in\dd,\ll\subsetneq L\subset\kk}\frac{1}{|L|}=\frac{1}{|\ll|}\left(1-\frac{|\ll|}{|\kk|}\right).
\end{equation}
These results follows from the fact that for every $k\geq 1$ there is a unique $L\supsetneq\ll$ with $|L|=2^k|\ll|$.

For $I\in\dd$, we denote by $h_I$ the corresponding Haar function, 
$$
h_I=\dst\frac{-\mathbf{1}_{I_-}+\mathbf{1}_{I_+}}{\sqrt{|I|}}.
$$
Note that, if $K\in\dd$ is such that $K\subset I_\pm$ then $h_I$ is constant on $K$.
Then, denoting by $c(K)$ the center of $K$, $h_I(K)=h_I\bigl(c(K)\bigr)=\frac{\eps(I,K)}{\sqrt{\left\vert I\right\vert}}$ where $\eps(I,K)\in\{\pm1\}$.  Also, $h_I$ has mean zero so that $\scal{\mathbf{1}_I,h_I}_{L^2}=0$ and, more generally, if $I\subset J$,
$\scal{\mathbf{1}_J,h_I}_{L^2}=0$.

Recall that $\{h_I: I\in\dd\}$ is an orthonormal basis of $L^2(\mathbb{R})$. In particular, if $f\in L^2(\R)$ and $I\in\dd$,
we write $\widehat{f}(I)=\scal{f,h_I}_{L^2}$ so that
$$
f=\sum_{I\in\dd}\widehat{f}(I)h_I
$$
and, for $f,g\in L^2(\R)$,
$$
\scal{f,g}_{L^2}=\sum_{I\in\dd}\widehat{f}(I)\overline{\widehat{g}(I)}.
$$
Further, when $f\in L^2(\R)$ is supported on an interval $I\in\dd$, then it is simpler to write
\begin{equation}
f  = \left\langle f\right\rangle_I \un_I+\sum_{J\subset I} \hat{f}(J)h_J
\label{eq:fsupp}
\end{equation}
from which it follows that
\begin{equation}
\left\Vert f\right\Vert_{L^2}^2  =  \left\langle f\right\rangle_{I}^2 \left\vert I\right\vert +\sum_{J\subset I} \left\vert \hat{f}(J)\right\vert^2
\label{eq:fsuppnorm}
\end{equation}
since $\un_I$ and $h_J$ are orthogonal when $J\subset I$.
On the other hand
\begin{equation}
\un_I=\sum_{L\in\dd}\scal{\un_I,h_L}_{L^2}h_L=\sum_{L\supsetneq I}\scal{\un_I,h_L}_{L^2}h_L=|I|\sum_{L\supsetneq I}h_L(I)h_L
\label{eq:harrun}
\end{equation}
since $\scal{\un_I,h_L}_{L^2}=\dst\int_I h_L(x)\d x=0$ when $L\subset I$.

Let $\Sha$ denote the \emph{dyadic Hilbert transform} (the Haar shift) which is the bounded linear operator on $L^2(\R)$
defined by
$$
\Sha h_I=\frac{h_{I_+}-h_{I_-}}{\sqrt{2}}.
$$
Note that $\Sha h_I$ is supported on $I$. It is easily seen that $\scal{\Sha h_I,\Sha h_J}_{L^2}=\delta_{I,J}$ so that
$\Sha $ is a \emph{unitary} transform.

We will now make a few simple observations.  

\begin{enumerate}
\item If $K$ is any dyadic interval than the function $\mathbf{1}_K \Sha h_L$ is supported on $K\cap L$.
In particular, if $L\subset K$, $\mathbf{1}_K \Sha h_L=\Sha h_L$.

\item If $L\supsetneq \widehat{K}$, then the function $\mathbf{1}_K \Sha h_L=\dst\frac{\eps(K,L)}{\sqrt{|L|}}\textbf{1}_K$
where $\eps(K,L)\in\{\pm 1\}$. We will write $\mathbf{1}_K \Sha h_L=\Sha h_L(K)\textbf{1}_K$
where again $\Sha h_L(K)=\Sha h_L\bigl(c(K)\bigr)$.\\
Indeed, $K=\widehat{K}_{\eps(K)}\subsetneq L$ thus $\widehat{K}_{\eps(K)}\subset L_\pm$
but then
$$
\mathbf{1}_K \Sha h_L=\pm\mathbf{1}_K\frac{h_{L_\pm}}{\sqrt{2}}=\pm\frac{h_{L_\pm}(K)}{\sqrt{2}}\mathbf{1}_K
$$
which is of the desired form.

\item If $L=\widehat{K}$, then $K=L_{\eps(K)}$ and $\mathbf{1}_K \Sha h_L=\dst\frac{\eps(K)h_K}{\sqrt{2}}$.
\end{enumerate}


When $f\in L^2(\R)$ is supported in $I\in\dd$, from the decomposition \eqref{eq:fsupp}, we obtain
\begin{equation}
\un_K\Sha f = \left\langle f\right\rangle_{I} \un_K\Sha \un_I +\sum_{J\subset I} \hat{f}(J) \un_K\Sha h_J.
\label{eq:indKH}
\end{equation}

On the other hand, from the decomposition \eqref{eq:harrun}, we have that for any $I,K\in\mathcal{D}$:
$$
\Sha \un_{I}  =  \left\vert I\right\vert \sum_{L\supsetneq I} h_{L}(I) \Sha h_L
$$
thus
\begin{equation}
\un_K \Sha \un_{I}  =  \left\vert I\right\vert \sum_{L\supsetneq I} h_{L}(I) \un_{K} \Sha  h_L .\label{e:restrictedHilbert}
\end{equation}

We can now prove the following

\begin{lm}\label{1KH1I:KinI}
For $I\in\dd$, $\un_I\Sha\un_I=\sqrt{|I|}h_I$.
\end{lm}

\begin{proof}
Let $K=I_\pm$. We want to prove that
$$
\un_{I_\pm} \Sha \un_{I}=\pm\un_{I_\pm}.
$$

From \eqref{e:restrictedHilbert}, we deduce that
$$
\un_K \Sha \un_{I} = \left\vert I\right\vert \left[\sum_{L\supsetneq I} h_{L}(I) \Sha  h_L(K)\right]\un_K.
$$
since $L\supsetneq\widehat{K}$ for any $I\subsetneq L$. Observe that the sign of $h_L(I)\Sha h_L(K)$, $I\subsetneq L$, only depends on the position of $K$ regarding $I_-$ or $I_+$. Indeed, if we have $K= I_-$ and $I\subset L_-$ then 
$\dst h_L(I)\Sha h_L(K)=\frac{-1}{|L|}$ with $\dst h_L(I)=\frac{-1}{\sqrt{|L|}}=-\Sha h_L(K)$ since $K= I_-\subset (L_-)_-$. 
On the other hand, if $K= I_-$ and $I\subset L_+$ then $\dst h_L(I)\Sha h_L(K)=\frac{-1}{|L|}$ with 
$\dst h_L(I)=\frac{1}{\sqrt{|L|}}=-\Sha h_L(K)$ since $K\subset I_-\subset (L_+)_-$. 
Similar arguments lead to $\dst h_L(I)\Sha h_L(K)=\frac{1}{|L|}$ when $K= I_+$ and $I\subset L_-$ and when $K= I_+$ and $I\subset L_+$. Thus, we obtain  
\begin{eqnarray*}
\un_K \Sha \un_{I}&=&\varepsilon(K,I)\left\vert I\right\vert\left[\sum_{L\supsetneq I}\frac{1}{\left\vert L\right\vert}\right] \un_K\\
                       &=&\varepsilon(K,I)\left\vert I\right\vert\left[\sum_{k=1}^{\infty} \frac{1}{2^k|I|}\right]\un_K\\
                       &=&\varepsilon(K,I)\un_K
\end{eqnarray*}
as announced.
\end{proof}

Our aim is to obtain lower bounds of $\norm{\un_K\Sha f}_2$ when $f\in L^2(\R)$ is supported in $I\in\dd$.
This requires an understanding of $\un_K \Sha \un_{I}$ in the three cases $K\subset I$, $I\subset K$ and $K\cap I=\emptyset$.

\section{First case: $I\subset K$}
\label{s3}

This is the ``easy'' and most favorable case:

\begin{thm}
Let $I\subset K\in\dd$. Then, for every $f\in L^2(\R)$ supported in $I$,
$$
\norm{\un_K\Sha f}_{L^2}^2\geq\left(1-\frac{3}{4}\frac{|I|}{|K|}\right)\norm{f}_{L^2}^2.
$$ 
\end{thm}

\begin{proof}
According to \eqref{e:restrictedHilbert} we have 
\begin{eqnarray}
\un_K\Sha f&=&\left\langle f\right\rangle_I \un_K\Sha \un_I +\sum_{J\subset I} \widehat{f}(J) \un_K\Sha h_J
\label{eq:unkhf1}\\
&=&\left\langle f\right\rangle_I \un_K\Sha \un_I +\sum_{J\subset I} \widehat{f}(J) \Sha h_J.\nonumber
\end{eqnarray}
Indeed, notice that in \eqref{eq:unkhf1}, $J\subset I\subset K$ so that $\Sha h_J$ is supported in $J\subset K$
and $\un_K\Sha h_J=\Sha h_J$.

Now we further have that:
\begin{eqnarray}
\norm{\un_K\Sha f}_{L^2}^2 & = & \scal{\un_K\Sha f, \un_K \Sha f}_{L^2}\nonumber\\
 & = & \scal{\left\langle f\right\rangle_I \un_K\Sha \un_I +\sum_{J\subset I} \hat{f}(J) \Sha h_J, 
\left\langle f\right\rangle_I \un_K\Sha \un_I +\sum_{J\subset I} \hat{f}(J) \Sha h_J}_{L^2}\nonumber\\
 & = & \left\langle f\right\rangle_{I}^2 \norm{\un_K \Sha \un_I}_{L^2}^2+ 
\scal{ \Sha \left(\sum_{J\subset I} \hat{f}(J) h_J\right), \Sha \left(\sum_{J\subset I} \hat{f}(J) h_J\right)}_{L^2}\nonumber\\
 & & +2\left\langle f\right\rangle_{I} \sum_{J\subset I}\scal{ \Sha \un_I, \Sha h_J}_{L^2} \hat{f}(J).
\label{eq:1kHf2}
\end{eqnarray}
But, as $\Sha $ is unitary, $\scal{\Sha 1_I, \Sha h_J}_{L^2}=\scal{\un_I, h_J}_{L^2}=0$
since $J\subset I$.
Further, using again that $\Sha$ is unitary and that the $\{h_J\}$'s are orthonormal,
\begin{eqnarray*}
\left\langle \Sha \left(\sum_{J\subset I} \hat{f}(J) h_J\right), \Sha \left(\sum_{J\subset I} \hat{f}(J) h_J\right)\right\rangle_{L^2}
&=&\left\langle \sum_{J\subset I} \hat{f}(J) h_J,\sum_{J\subset I} \hat{f}(J) h_J\right\rangle_{L^2}\\
&=&\sum_{J\subset I} |\hat{f}(J)|^2.
\end{eqnarray*}
Therefore \eqref{eq:1kHf2} reduces to
$$
\norm{\un_K\Sha f}_{L^2}^2=\left\langle f\right\rangle_{I}^2 \left\Vert \un_K \Sha \un_I\right\Vert_{L^2}^2
+\sum_{J\subset I} |\hat{f}(J)|^2.
$$
As $\norm{\un_K\Sha \un_I}_{L^2}^2\leq\norm{\Sha \un_I}_{L^2}^2=|I|$, we get
\begin{equation}
\label{eq:lbunkHf}
\norm{\un_K\Sha f}_{L^2}^2\geq\frac{\norm{\un_K\Sha \un_-I}_{L^2}^2}{|I|}\left(\left\langle f\right\rangle_{I}^2\vert I\vert
+\sum_{J\subset I} |\hat{f}(J)|^2\right)
=\frac{\norm{\un_K\Sha \un_I}_{L^2}^2}{|I|}\norm{f}^2_{L^2}.
\end{equation}

It remains to estimate $\norm{\un_K\Sha \un_I}_{L^2}^2$ from below. Recall form \eqref{e:restrictedHilbert}
that
\begin{eqnarray}
\frac{1}{|I|}\un_K \Sha \un_{I} & = &\sum_{L\supsetneq I} h_{L}(I) \un_{K} \Sha  h_L 
= \left(\sum_{L\supsetneq \widehat{K}}+\sum_{L=\widehat{K}}+\sum_{K\supset L\supsetneq I}\right) h_{L}(I) \un_{K} \Sha  h_L \nonumber\\
&=&\left(\sum_{L\supsetneq \widehat{K}} h_{L}(I) \Sha  h_L(K)\right)\un_{K}
+\frac{\eps(K) h_{\hat K}(I)}{\sqrt{2}}h_K
+\sum_{K\supset L\supsetneq I} h_L(I) \Sha h_L\label{eq:decomp}
\end{eqnarray}
with the three observations made on $\un_{K} \Sha h_L$.
Now notice that the three terms in \eqref{eq:decomp} are orthogonal. Indeed, if $L\subset K$
then $h_K$ and $\Sha h_L$ are supported in $K$ and have mean $0$. Therefore, they are orthogonal
to $\un_K$. Further,  $\sqrt{2}\Sha h_L=h_{L_+}-h_{L_-}$ and $L_\pm\subsetneq K$ thus
$h_{L_\pm}$ is orthogonal to $h_K$. Moreover,
$$
\abs{\frac{\eps(K) h_{\hat K}(I)}{\sqrt{2}}}=\frac{1}{\sqrt{2|\hat K|}}
=\frac{1}{2\sqrt{|K|}}
$$
and, as $\Sha $ is unitary, the $\Sha h_L$'s are orthonormal.
Therefore
\begin{eqnarray*}
\frac{\norm{\un_K\Sha \un_I}_{L^2}^2}{|I|}
&=&|I||K|\left(\sum_{L\supsetneq \widehat{K}} h_{L}(I) \Sha  h_L(K)\right)^2
+\frac{|I|}{4|K|}+|I|\sum_{K\supset L\supsetneq I} |h_L(I)|^2\\
&\geq&\frac{|I|}{4|K|}+|I|\sum_{K\supset L\supsetneq I}\frac{1}{|L|}.
\end{eqnarray*}
Now this last quantity is $\dst\frac{1}{4}=1-\frac{3}{4}\frac{|I|}{|K|}$ when $K=I$ and 
\eqref{eq:freqcomp2} shows that it is
$\dst 1-\frac{3}{4}\frac{|I|}{|K|}$
when $K\supsetneq I$, which completes the proof.
\end{proof}

\section{Second case: $I\cap K=\emptyset$}
\label{s4}

Suppose that $K,I\in\mathcal{D}$ are such that $K\cap I=\emptyset$.  First observe that
\begin{eqnarray*}
\un_K\Sha f & = & \left\langle f\right\rangle_I \un_K \Sha \un_I+\sum_{J\subset I} \hat{f}(J) \un_K \Sha h_J\\
 & = & \left\langle f\right\rangle_I \un_K \Sha \un_I
\end{eqnarray*}
with the last equality following since $\Sha h_J$ is supported on $J\subset I$ and that $I\cap K=\emptyset$ and so $J\cap K=\emptyset$ as well.  Thus, we have that
$$
\left\Vert \un_K\Sha f\right\Vert_{L^2}^2 =\frac{\left\Vert \un_K \Sha \un_I\right\Vert_{L^2}^2}{\left\vert I\right\vert} \left\langle f\right\rangle_{I}^2\left\vert I\right\vert.
$$

\begin{rem*}
From this, it is obvious that a lower bound of the form 
$\left\Vert \un_K\Sha f\right\Vert_{L^2}^2\geq C\left\Vert f\right\Vert_{L^2}^2
=\left\langle f\right\rangle_I^2\left\vert I\right\vert+\sum_{J\subset I}\left\vert \hat{f}(I)\right\vert^2$
cannot hold without further assumptions on $f$. For instance, if $f$ has mean $0$ then $\un_K\Sha f=0$.
One may also restrict attention to non-negative functions in which case the mean would not be zero.
However, $\dst \sum_{J\subset I}\left\vert \hat{f}(I)\right\vert^2$ may still be arbitrarily large
compared to $\left\langle f\right\rangle_I^2\left\vert I\right\vert$ so that
we would still not obtain a bound of the form $\left\Vert \un_K\Sha f\right\Vert_{L^2}^2\geq C\left\Vert f\right\Vert_{L^2}^2$.

One way to overcome this is to ask for a restriction on the oscillations of $f$. For example, when $f$ is in the Sobolev 
space $W^{1,2}(I)$ and $f'$ its derivative. We extend both $f$ and $f'$ by $0$ outside the interval $I$ (so that $f'$ needs not be the 
distributional derivative of $f$ over $\R$). Alternatively, $f'$ may be defined as the derivative of the Fourier series of $f$
and extended by $0$ outside $I$, see below.
By Poincar\'e-Wirtinger
({\it see e.g.} \cite[Chap 4]{Da} or \cite[Chap 5]{ABM}) we have that:
\begin{equation}
\label{eq:PW}
\left\Vert f -\left\langle f\right\rangle_I \un_I\right\Vert_{L^2(I)}\leq \frac{\left\vert I \right\vert}{2\pi}\left\Vert f'\right\Vert_{L^2(I)}.
\end{equation}
Now, suppose that the norm of the derivative is controlled relative to the norm of the function:
\begin{equation}
\label{eq:PWbis}
\left\Vert f'\right\Vert_{L^2(I)} \leq \eta \frac{2\pi \left\Vert f\right\Vert_{L^2(I)}}{\left\vert I\right\vert}, \quad 0\leq\eta<1,
\end{equation}
then we will have that:
\begin{eqnarray*}
\left\Vert f\right\Vert_{L^2(I)} & \leq & \left\Vert f -\left\langle f\right\rangle_I \un_I\right\Vert_{L^2(I)}+\left\vert I\right\vert^{\frac{1}{2}}\left\vert \left\langle f\right\rangle_I\right\vert\\
& \leq & \eta \left\Vert f\right\Vert_{L^2(I)}+\left\vert I\right\vert^{\frac{1}{2}}\left\vert \left\langle f\right\rangle_I\right\vert,
\end{eqnarray*}
which upon rearrangement will give 
$$
\left\vert I\right\vert  \left\langle f\right\rangle_I^2\geq (1-\eta)^2
\left\Vert f\right\Vert_{L^2(I)}^2.
$$
In other words, function satisfying \eqref{eq:PW} are small zero-mean perturbations of constants. 
For instance, with $I=[0,1]$, let $(a_k)_{k\in\Z\setminus\{0\}}$ be a sequence such that $\alpha^2:=\sum_{k\not=0}k^2|a_k|^2<+\infty$,
and $a_0=\frac{\alpha}{\eta}$. We may then define $f(t)=\sum_{k\in\Z}a_ke^{2ik\pi t}$ on $[0,1]$ where the series converges uniformly and extend $f$ by $0$ outside $[0,1]$. On $[0,1]$ the weak derivative of $f$ is given by
$f'(t)=\sum_{k\in\Z}2ik\pi a_ke^{2ik\pi t}$ where the sum of the series is taken in the $L^2([0,1])$ sense (and needs not be extended outside $[0,1]$). It follows that $f$ satisfies \eqref{eq:PWbis}.

One can replace the Poincar\'e-Wirtinger inequality by versions where one tests the $L^p$ norm of the derivative and the $L^2$ norm
of the function. For such inequalities, we refer to \cite[Chap 5]{ABM}.
\end{rem*}

We now turn to computing a lower bound of $\frac{\left\Vert \un_K \Sha \un_I\right\Vert_{L^2}^2}{\left\vert I\right\vert}$.  
First, $\Sha \un_I$ is supported in $I$ so that $\un_K \Sha \un_I=0$ if $K\subset \R^\pm$ and $I\subset \R^\mp$.
We will therefore assume that $I,K\subset \R^+$, the case $I,K\subset\R^-$ then follows from the fact that
$\Sha $ is ``odd'', thus $\un_K \Sha \un_I=-\un_{-K} \Sha \un_{-I}$.

Let $K\wedge I$ denote the minimal dyadic interval that contains both $K$ and $I$. 
Note that $I,K\not=K\wedge I$, so that $I$ and $K$ belong to different dyadic children of $K\wedge I$; for example if $I\subset(K\wedge I)_+$ then $K\subset(K\wedge I)_-$ and a similar statement holdis when replacing the appropriate $+$ and $-$.
Let us now split the identity \eqref{e:restrictedHilbert} into three parts 
\begin{eqnarray}
\frac{\un_K \Sha \un_I}{|I|} & = &  \sum_{L\supsetneq I} h_L(I) \un_K \Sha h_L\nonumber\\
& =&   \left(\sum_{L\supsetneq K\wedge I}+\sum_{L=K\wedge I}+\sum_{K\wedge I\supsetneq L\supsetneq I}\right) h_L(I) \un_K \Sha h_L\nonumber\\
\label{eq:3rdterm} & = & \left[\,\sum_{L\supsetneq K\wedge I} h_L(I) \Sha h_L(K) \right]\un_{K}
+h_{K\wedge I}(I)\un_{K} \Sha h_{K\wedge I} +\sum_{K\wedge I\supsetneq L\supsetneq I}h_L(I) \un_K \Sha h_L 
\end{eqnarray}
since we have that $1_K \Sha h_L$ takes a constant value as described above when $L\supsetneq K\wedge I$ and evaluating the sums over the regions in question.

Let us now notice that $L\cap K=\emptyset$ when $I\subsetneq L\subsetneq K\wedge I$. Indeed, suppose this were not the case.
It is not possible that $L\subset K$ since $I\subset L\subset K$, which contradicts that $I\cap K=\emptyset$.  Thus we have that $I,K\subset L$ and hence $K\wedge I\subset L$, contradicting that $L\subsetneq K\wedge I$, and so $L\cap K=\emptyset$ as claimed.
It follows that the third term in \eqref{eq:3rdterm} vanishes so that
\begin{eqnarray*}
\frac{\un_K \Sha \un_I}{|I|} & = &
\left[\,\sum_{L\supsetneq K\wedge I} h_L(I) \Sha h_L(K) \right]\un_{K}+h_{K\wedge I}(I)\un_{K} \Sha h_{K\wedge I}\\
& = & \begin{cases}
\dst\left[\,\sum_{L\supsetneq K\wedge I} h_L(I) \Sha h_L(K) \right]\un_{K}+\varepsilon(K) \frac{h_{K\wedge I}(I) h_{K}}{\sqrt{2}}
&\mbox{if }K\wedge I=\widehat{K}\\[18pt]
\dst \left[\,\sum_{L\supset K\wedge I} h_L(I) \Sha h_L(K) \right]\un_{K}&\mbox{if }K\wedge I \supsetneq \widehat{K}
\end{cases}
\end{eqnarray*}
which follows from the properties of $\un_K\Sha h_L$ given above.

Thus, we have that:
$$
\frac{\left\Vert \un_K \Sha \un_I\right\Vert_{L^2}^2}{\left\vert I\right\vert} = 
\begin{cases}
\dst\left\vert I\right\vert \left\vert K\right\vert \left\vert \sum_{L\supsetneq K\wedge I} h_L(I) \Sha h_L(K) \right\vert^2 
+\frac{\left\vert I\right\vert }{2\left\vert K\wedge I\right\vert} &\mbox{if }K\wedge I=\widehat{K}\\[18pt]
\dst\left\vert I\right\vert \left\vert K\right\vert \left\vert \sum_{L\supset K\wedge I} h_L(I) \Sha h_L(K)\right\vert^2
&\mbox{if }K\wedge I \supsetneq \widehat{K}.
\end{cases}
$$

\begin{rem*}
At this stage, we can observe that, when $K\wedge I\subsetneq\widehat{K}$,
$$
\frac{\left\Vert \un_K \Sha \un_I\right\Vert_{L^2}^2}{\left\vert I\right\vert}\leq \frac{1}{4}.
$$
Indeed, we have that 
\begin{eqnarray*}
\frac{\left\Vert \un_K \Sha \un_I\right\Vert_{L^2}^2}{\left\vert I\right\vert} & \leq & 
\left\vert I\right\vert \left\vert K\right\vert \left(\sum_{L\supsetneq K\wedge I} \left\vert h_L(I)\right\vert \left\vert \Sha h_L(K)\right\vert\right)^2\\
& = & \left\vert I\right\vert \left\vert K\right\vert \left(\sum_{L\supsetneq K\wedge I} \frac{1}{\sqrt{\left\vert L\right\vert}} 
\frac{1}{\sqrt{\left\vert L\right\vert}} \right)^2\\
& = & \left\vert I\right\vert \left\vert K\right\vert \left(\sum_{L\supsetneq K\wedge I} \frac{1}{\left\vert L\right\vert} \right)^2\\
& = & \frac{\left\vert I\right\vert \left\vert K\right\vert}{\left\vert K\wedge I\right\vert^2}\leq \frac{1}{4}.
\end{eqnarray*}
Here the last inequality follows since $I,K\subsetneq K\wedge I$,
so $\left\vert I\right\vert,\left\vert K\right\vert \leq\frac{1}{2} \left\vert K\wedge I\right\vert$.

If $K\wedge I=\widehat{K}$, there is an extra term and we get $\dst\frac{|I|}{2|K\wedge I|}\leq\frac{1}{4}$
from which we deduce that
$$
\frac{\left\Vert \un_K \Sha \un_I\right\Vert_{L^2}^2}{\left\vert I\right\vert}\leq \frac{1}{2}.
$$
\end{rem*}

Note that, if $K\wedge I = \widehat{K}$, then we write $K=K_-\cup K_+$ so that $K_\pm\wedge I=\widehat{K}$
and $\un_K\Sha\un_I=\un_{K_-}\Sha\un_I+\un_{K_+}\Sha\un_I$ is an orthogonal decomposition.

To give an estimation of $\left\vert \sum_{L\supset K\wedge I} h_L(I) \Sha h_L(K)\right\vert^2$ when 
$K\wedge I \supsetneq \widehat{K}$, we use the following lemma.

\begin{lm}\label{lem:sign}
Let $\ll^{(0)}=\ll:=K\wedge I$ and for $k\geq 1$, $\ll^{(k)}=\widehat{\ll^{(k-1)}}$. Let $\varepsilon(K)$ be equal to $1$ if $K\subset\ll_+$ and $-1$ if $K\subset\ll_-$. Then, we have 

\begin{enumerate}
\renewcommand{\theenumi}{\roman{enumi}}
\item $h_{\ll}(I)\Sha h_\ll(K)=\begin{cases}
\frac{-1}{|\ll|}&\hbox{if } K\subset(\ll_+)_+\\
\frac{1}{|\ll|}&\hbox{if } K\subset(\ll_+)_-\\
\frac{-1}{|\ll|}&\hbox{if } K\subset(\ll_-)_+\\
\frac{1}{|\ll|}&\hbox{if } K\subset(\ll_-)_-;
\end{cases}$

\item $h_{\ll^{(1)}}(I)\Sha h_{\ll^{(1)}}(K)=\frac{\varepsilon(K)}{2|\ll|}$;


\item and for $k\geq 2$, $h_{\ll^{(k)}}(I)\Sha h_{\ll^{(k)}}(K)=\begin{cases}
\frac{1}{2^k|\ll|}&\hbox{if } \ll^{(k-2)}=\ll^{(k-1)}_+\\
-\frac{1}{2^k|\ll|}&\hbox{if } \ll^{(k-2)}=\ll^{(k-1)}_-\,.
\end{cases}$
\end{enumerate}

\end{lm}

\begin{proof}
It is enough to deal with the case $K\subset\ll_+$ (i.e. $\varepsilon(K)=1$). Since $I\cap K=\emptyset$ and by 
the definition of $\ll$, we have $I\subset\ll_-$ and $h_\ll(I)=\frac{-1}{\sqrt{|\ll|}}$. Now, there are only two 
cases to consider for $K$: either $K\subset(\ll_+)_+$ and $\Sha h_\ll(K)=\frac{1}{\sqrt{|\ll|}}$ or 
$K\subset(\ll_+)_-$ and $\Sha h_\ll(K)=\frac{-1}{\sqrt{|\ll|}}$. It follows that 
$$
h_{\ll}(I)\Sha h_\ll(K)=\begin{cases}
-\frac{1}{|\ll|}&\hbox{if } K\subset(\ll_+)_+\\
\frac{1}{|\ll|}&\hbox{if } K\subset(\ll_+)_-.
\end{cases}
$$

Suppose first that $\ll=\ll^{(1)}_+$. Then, we have $I\subset\ll=\ll^{(1)}_+$ and $K\subset \ll_+=(\ll^{(1)}_+)_+$
 which implies that $h_{\ll^{(1)}}(I)\Sha h_{\ll^{(1)}}(K)=\frac{1}{|\ll^{(1)}|}$ 
with $h_{\ll^{(1)}}(I)=\Sha h_{\ll^{(1)}}(K)=\frac{1}{\sqrt{|\ll^{(1)}|}}$.
On the other hand, if $\ll=\ll^{(1)}_-$ then we have $I\subset\ll=\ll^{(1)}_-$ and $K\subset \ll_+=(\ll^{(1)}_-)_+$. 
We still obtain that $h_{\ll^{(1)}}(I)\Sha h_{\ll^{(1)}}(K)=\frac{1}{|\ll^{(1)}|}$ with 
$h_{\ll^{(1)}}(I)=\Sha h_{\ll^{(1)}}(K)=\frac{-1}{\sqrt{|\ll^{(1)}|}}$.

\smallskip

Let us prove property $(iii)$ for $k\geq 2$. Suppose first that $\ll^{(k-2)}=\ll^{(k-1)}_+$. 
When $\ll^{(k-1)}=\ll^{(k)}_+$, we have that $I\subset\ll^{(k-1)}=\ll^{(k)}_+$ and $K\subset \ll^{(k-1)}_+=(\ll^{(k)}_+)_+$
which implies that $h_{\ll^{(k)}}(I)\Sha h_{\ll^{(k)}}(K)=\frac{1}{|\ll^{(k)}|}$ with 
$h_{\ll^{(k)}}(I)=\Sha h_{\ll^{(k)}}(K)=\frac{1}{\sqrt{|\ll^{(k)}|}}$. And, when $\ll^{(k-1)}=\ll^{(k)}_-$, we have that 
$I\subset\ll^{(k-1)}=\ll^{(k)}_-$ and $K\subset \ll^{(k-1)}_+=(\ll^{(k)}_-)_+$ which implies that 
$h_{\ll^{(k)}}(I)\Sha h_{\ll^{(k)}}(K)=\frac{1}{|\ll^{(k)}|}$ with 
$h_{\ll^{(k)}}(I)=\Sha h_{\ll^{(k)}}(K)=\frac{-1}{\sqrt{|\ll^{(k)}|}}$.
One can easily deduce the case $\ll^{(k-2)}=\ll^{(k-1)}_-$ which leads to 
$h_{\ll^{(k)}}(I)\Sha h_{\ll^{(k)}}(K)=\frac{-1}{|\ll^{(k)}|}$.
\end{proof}

Let us now prove the first sub-case.

\begin{lm}
We suppose that $K$, $I\subset\R_+$, $K\cap I=\emptyset$. Let $\ll=K\wedge I$ and assume that $\ll=[0,2^{N})$ for some $N\in\Z$.
\begin{enumerate}
\item Assume that $I\subset\ll_+$ while $K\subsetneq\ll_-$.  Then
\begin{enumerate}
\item If $K\subset\ll_{--}$ then $\un_K \Sha \un_I=0$ thus $\dst\frac{\|\un_K \Sha \un_I\|^2}{|I|}=0$;

\item If $K\subset\ll_{-+}$ then $\un_K \Sha \un_I=-\frac{2|I|}{|\ll|}\un_K$ thus 
$\dst\frac{\|\un_K \Sha \un_I\|^2}{|I|}=4\frac{|I||K|}{|\ll|^2}$.
\end{enumerate}

\item Assume that $I\subset\ll_-$ while $K\subset\ll_{+\pm}$. Then
$\un_K \Sha \un_I=\pm\frac{|I|}{|\ll|}\un_K$thus 
$\dst\frac{\|\un_K \Sha \un_I\|^2}{|I|}=\frac{|I||K|}{|\ll|^2}$.
\end{enumerate}
\end{lm}

\begin{proof}
Now let again $\ll^{(k)}$ be defined by $\ll^{(0)}=\ll$ and $\ll^{(k+1)}=\widehat{\ll^{(k)}}$. 
Note that, as $\ll=[0,2^{N^0})$, $\ll^{(k)}=\ll^{(k+1)}_-$.
As $\widehat{K}=\ll_\pm\not=\ll$, we want to estimate 
$$
\frac{1}{|I|}\un_K \Sha \un_I=
\left(\sum_{L\supset \ll} h_L(I) \Sha h_L(K)\right)\un_K=\left(\sum_{k\geq 0}h_{\ll^{(k)}}(I)\Sha h_{\ll^{(k)}}(K)\right)\un_K.
$$
Assume first that $K\subset\ll_{--}$ and $I\subset\ll_+$ . Then, according to the previous lemma,
$$
h_{\ll^{(0)}}(I)\Sha h_{\ll^{(0)}}(K)=\frac{1}{|\ll|}
\quad\mbox{while}\quad
h_{\ll^{(k)}}(I)\Sha h_{\ll^{(k)}}(K)=\frac{-1}{2^k|\ll|}
$$
for $k\geq 1$. The result follows immediately.

Assume now that $K\subset(\ll_{-})_{+}$ and $I\subset\ll_+$ . Then, according to the previous lemma again,
$$
h_{\ll^{(k)}}(I)\Sha h_{\ll^{(k)}}(K)=\frac{-1}{2^k|\ll|}
$$
for $k\geq 0$. The result again follows immediately.

Let us now assume that $K\subset(\ll_{+})_{\pm}$ and $I\subset\ll_-$. Then, according to the previous lemma,
$$
h_{\ll^{(0)}}(I)\Sha h_{\ll^{(0)}}(K)=\frac{\pm1}{|\ll|}
$$
while
$$
h_{\ll^{(1)}}(I)\Sha h_{\ll^{(1)}}(K)=\frac{1}{2|\ll|}
\quad\mbox{and}\quad
h_{\ll^{(k)}}(I)\Sha h_{\ll^{(k)}}(K)=\frac{-1}{2^k|\ll|}\qquad k\geq2
$$
and the result again follows immediately.
\end{proof}

Now if $I\subset\dd$, there exists $M_0$ such that $I\subset[0,2^{M_0}]$ but $I\not\subset [0,2^{M_0-1}]$.
In the case $I=[0,2^{M_0}]$, the previous lemma determines $\Sha \un_I$ on $I^c$. 
Otherwise $I\subset[2^{M_0-1},2^{M_0}]$ and the previous
lemma determines $H\un_I$ on $[0,2^{M_0-1}]$ and on $[2^{M_0},+\infty)$.

It remains to consider the case $K,I$ such that $K\cap I=\emptyset$
and $K,I\subset [2^{M_0-1},2^{M_0}]$. We keep the same notation: $\ll=K\wedge I$ for the first common ancestor
of $K$ and $I$, $\ll^{(0)}=\ll$ and $\ll^{(k)}=\widehat{\ll^{(k-1)}}$ for $k\geq 1$.
We further write $\ll^*=[0,2^{M_0}]$ the first common ancestor of $K,I$ of the form $[0,2^M]$
so that $K\wedge I\subset\ll^*_+$ Let $k^*$ be defined by $\ll^*=\ll^{(k^*)}$. It follows that
$2^{M_0}=|\ll^*|=2^{k^*}|\ll|=2^{k^*}|K\wedge I|$. Now
\begin{eqnarray*}
\frac{1}{|I|}\un_K\Sha\un_I&=&
\un_K\sum_{L\supset\ll}h_L(I)\Sha h_L(K)\\
&=&\un_K\left(\sum_{L\supsetneq\ll^*}h_L(I)\Sha h_L(K)
+\sum_{\ll\subset L\subset\ll^*}h_L(I)\Sha h_L(K)\right)\\
&=&\un_K\sum_{\ll\subset L\subset\ll^*}h_L(I)\Sha h_L(K).
\end{eqnarray*}
Indeed, if $L=\widehat{\ll^*}=\ll^{(k^*+1)}$ then $\ll^{(k^*-1)}\subset\ll^{(k^*)}_+$
so that, according to Lemma \ref{lem:sign}, 
$$
h_L(I)\Sha h_L(K)=\frac{1}{2^{k^*+1}|\ll|}.
$$
On the other hand, if $L=\ll^{(k)}$ for $k\geq k^*+2$, $\ll^{(k-2)}\subset \ll^{(k-1)}_-$ so that
$$
h_L(I)\Sha h_L(K)=-\frac{1}{2^{k^*+1}|\ll|}.
$$
Therefore, $\dst\sum_{L\supsetneq\ll^*}h_L(I)\Sha h_L(K)=0$.

We now distinguish 2 cases. First assume that $\ll=\ll^*_+$. Then
$$
\frac{1}{|I|}\un_K\Sha\un_I=\un_K\bigl(h_{\ll^*_+}(I)\Sha h_{\ll^*_+}(K)+h_{\ll^*}(I)\Sha h_{\ll^*}(K)\bigr)\\
$$
Applying Lemma \ref{lem:sign} we get
$$
\frac{1}{|I|}\un_K\Sha\un_I=\begin{cases}-\frac{1}{2|\ll|}\un_K&\mbox{if }I\subset\ll_-,K\subset(\ll_{+})_{+}\\
\frac{3}{2|\ll|}\un_K&\mbox{if }I\subset\ll_-,K\subset(\ll_{+})_{-}\\
-\frac{3}{2|\ll|}\un_K&\mbox{if }I\subset\ll_+,K\subset(\ll_{-})_{+}\\
\frac{1}{2|\ll|}\un_K&\mbox{if }I\subset\ll_-,K\subset(\ll_{-})_{-}.
\end{cases}
$$
Let us now assume that $\ll\subsetneq\ll^*_+$. Then each $L$ with $\ll\subset L\subset\ll^*$ is of the form
$L=\ll^{(k)}$ with $0\leq k\leq k^*$ and for each such $k$, there is an $\eps_k=\pm1$ such that
$\dst h_L(I)\Sha h_L(K)=\frac{\eps_k}{2^k|\ll|}$. But then
\begin{eqnarray*}
\abs{\frac{1}{|I|}\un_K\Sha\un_I}&=&\un_K\abs{\sum_{k=0}^{k^*}\frac{\eps_k}{2^k|\ll|}}
=\frac{\un_K}{|\ll|}\left(1+\sum_{k=1}^{k^*}\frac{\eps_0\eps_k}{2^k}\right)\\
&\geq&\frac{\un_K}{|\ll|}\left(1-\sum_{k=1}^{k^*}2^{-k}\right)=\frac{\un_K}{|\ll|}\frac{|K\wedge I|}{2^{M_0}}
\end{eqnarray*}
so that 
$$
\frac{\norm{\un_K\Sha\un_I}^2_{L^2}}{|I|} \geq \left(\frac{|K\wedge I|}{2^{M_0}}\right)^2 \frac{|I||K|}{|\ll|^2}.
$$


We can now summarize the results of this section:

\begin{thm}
Let $\eta>0$. Let $I,K\in\dd $ be such that $I\subset\R^+$ and let $M_0$ be the smallest integer
such that $I\subset[0, 2^{M_0}]$.
Let $f_0\in W^{1,2}(I)$ be such that $|I|\norm{f_0^\prime}_{L^2(I)}\leq 2\pi\eta\norm{f_0}_{L^2(I)}$ and let $f$ be the extension of $f_0$ by $0$.  Then
\begin{enumerate}
\renewcommand{\theenumi}{\roman{enumi}}
\item If $K\subset \R_-$ then $\un_K\Sha f=0$.
\item If $K\subset [2^{M_0+k},2^{M_0+k+1}]$ then
$$
\norm{\un_K\Sha f}_{L^2}^2\geq (1-\eta)^2\frac{|I||K|}{2^{2(M_0+k)}}\norm{f}_{L^2}^2.
$$

\item If $I\subset [2^{M_0-1},2^{M_0}]$ then
\begin{enumerate}
\item If $K\subset [0,2^{M_0-2}]$ then $\un_K\Sha f=0$;

\item If $K\subset [2^{M_0-2},2^{M_0-1}]$ then
$$
\norm{\un_K\Sha f}_{L^2}^2\geq (1-\eta)^2\frac{|I||K|}{2^{2(M_0-1)}}\norm{f}_{L^2}^2;
$$
\item $K\subset [2^{M_0-1},2^{M_0}]$ and $K\cap I=\emptyset$ then
$$
\norm{\un_K\Sha f}_{L^2}^2\geq (1-\eta)^2\frac{|I||K||K\wedge I|^2}{2^{4M_0}}\norm{f}_{L^2}^2.
$$
\end{enumerate}
\end{enumerate}
In all of the above cases, no estimate of the form $\norm{\un_K\Sha f}_{L^2}^2\geq C\norm{f}_{L^2}^2$ can hold for all functions
$f\in L^2$ with support in $I$.
\end{thm}

\section{Third case: $K\subsetneq I$}
\label{s5}

For $K\subsetneq I$, we write write $\varepsilon(K,I)=+1$ if $K\subset I_+$ and $\varepsilon(K,I)=-1$ if $K\subset I_-$.
According to Lemma \ref{1KH1I:KinI}, $\un_K\Sha\un_I=\eps(K,I)\un_K$, in particular, $\Vert\un_K \Sha \un_{I}\Vert_{L^2}^2=\vert K\vert$.

From equation \eqref{eq:indKH} and Lemma \ref{1KH1I:KinI} we get that
\begin{eqnarray}
\un_K\Sha f&=&\langle f\rangle_{I}\un_K \Sha \un_I+\sum_{K\subsetneq J\subset I}\widehat{f}(J) \un_K \Sha h_J
+\sum_{J\subset K}\widehat{f}(J) \un_K \Sha h_J\nonumber\\
 &=&\left[\langle f\rangle_{I}\varepsilon(K,I) +\sum_{\widehat{K}\subsetneq J\subset I}\widehat{f}(J) \Sha h_J(K)\right]\un_K
+\frac{\varepsilon(K)}{\sqrt{2}}\widehat{f}(\widehat{K})h_K +\sum_{J\subset K}\widehat{f}(J) \Sha h_J.\label{eq:unkhf}
\end{eqnarray}
Let us denote by $\bb$ the subspace $\overline{\hbox{span}}\left\{\Sha h_J,J\subset K\right\}$ and $P_\bb$ the orthogonal projection onto $\bb$. Observe that for $J\subset K$, $\langle \un_K \Sha f,\Sha h_J\rangle=\langle \Sha f,\Sha h_J\rangle=\widehat{f}(J)$. Therefore 
\begin{equation}
\label{eq:PB}
P_\bb(\un_K \Sha f)=\sum_{J\subset K}\langle \un_K \Sha f,\Sha h_J\rangle \Sha h_J=\sum_{J\subset K}\widehat{f}(J) \Sha h_J.
\end{equation}
Moreover, the $h_J$'s being orthonormal and $\Sha $ being unitary,
\begin{equation}
\label{eq:normPB}
\norm{P_\bb(\un_K \Sha f)}_{L^2}^2=\sum_{J\subset K}|\widehat{f}(J)|^2.
\end{equation}

On the other hand, from \eqref{eq:unkhf} and \eqref{eq:PB}, it follows that 
$$
(I-P_\bb)(\un_K\Sha f)=\left[\langle f\rangle_{I}\varepsilon(K,I) +\sum_{\widehat{K}\subsetneq J\subset I}\widehat{f}(J) \Sha h_J(K)\right]\un_K+\frac{\varepsilon(K)}{\sqrt{2}}\widehat{f}(\widehat{K})h_K.
$$
But $h_K$ and $\un_K$ are orthogonal so that
\begin{equation}
\label{eq:normPBc}
\norm{(I-P_\bb)(\un_K \Sha f)}_{L^2}^2
=\left[\langle f\rangle_{I}\varepsilon(K,I) 
+\sum_{\widehat{K}\subsetneq J\subset I}\widehat{f}(J) \frac{\varepsilon(K,J)}{\sqrt{\vert J\vert}}\right]^2\vert K\vert
+\frac{\left\vert\widehat{f}(\widehat{K})\right\vert^2}{2}.
\end{equation}

We can now prove the following:

\begin{thm} Let $I,K\in\dd$ be such that $K\subset I$. Then, for every $f\in L^2(\R)$ with $\supp f\subset I$,
\begin{equation}
\label{eq:propfinal}
\Vert \un_K\Sha f\Vert_{L^2}^2
=\left[\langle f\rangle_{I}\varepsilon(K,I) +\sum_{\widehat{K}\subsetneq J\subset I}\widehat{f}(J) \frac{\varepsilon(K,J)}{\sqrt{\vert J\vert}}\right]^2\vert K\vert+\frac{\left\vert\widehat{f}(\widehat{K})\right\vert^2}{2}+\sum_{J\subset K}\left\vert\widehat{f}(J)\right\vert^2.
\end{equation}
In particular,
\begin{enumerate}
\renewcommand{\theenumi}{\roman{enumi}}
\item\label{eq:normkhf} for every $f\in L^2(\R)$, $\Vert \un_K\Sha f\Vert_{L^2}^2\geq \Vert \un_Kf\Vert_{L^2}^2$ and
$\dst \Vert \un_K\Sha f\Vert_{L^2}^2\geq \frac{1}{2}\Vert \un_{\hat K}f\Vert_{L^2}^2$.

\item If $I\supsetneq\widehat{K}$, there exists no constant $C=C(K,I)$ such that,
for every $f\in L^2(\R)$ with $\supp f\subset I$, $\Vert \un_K\Sha f\Vert_{L^2}\geq C \Vert f\Vert_{L^2}$.
\end{enumerate}
\end{thm}

\begin{proof} As $\norm{\un_K \Sha f}_{L^2}^2=\norm{P_\bb(\un_K \Sha f)}_{L^2}^2+\norm{(I-P_\bb)(\un_K \Sha f)}_{L^2}^2$,
\eqref{eq:propfinal} is a direct combination of \eqref{eq:normPB} and \eqref{eq:normPBc}.
The inequalities \eqref{eq:normkhf} are direct consequences of \eqref{eq:propfinal}.

For the last part of the proposition, let $f=\dst-\frac{\eps(K,I)}{\sqrt{|I|}}\un_I+h_I$. Then
$f\in L^2(\R)$ is supported in $I$ and
$\widehat{f}(J)=\delta_{I,J}$ if $J\subset I$ and $\langle f\rangle_{I}=\dst-\frac{\eps(K,I)}{\sqrt{|I|}}$.
Further \eqref{eq:propfinal} shows that $\Vert \un_K\Sha f\Vert_{L^2}=0$ while
$\Vert f\Vert_{L^2}=\sqrt{2}$.
\end{proof}


\section*{Acknowledgments}
The first author kindly acknowledge financial support from the French ANR program, ANR-12-BS01-0001 (Aventures),
the Austrian-French AMADEUS project 35598VB - ChargeDisq, the French-Tunisian CMCU/UTIQUE project 32701UB Popart.
This study has been carried out with financial support from the French State, managed
by the French National Research Agency (ANR) in the frame of the Investments for
the Future Program IdEx Bordeaux - CPU (ANR-10-IDEX-03-02).

Research supported in part by a National Science Foundation DMS grants DMS \# 1603246 and \# 1560955.  This research was partially conducted while B. Wick was visiting Universit\'e de Bordeaux as a visiting CNRS researcher. He thanks both institutions for their hospitality.

\end{document}